\documentclass[11pt]{amsart}
\usepackage{amsfonts,amssymb,amsthm}
\usepackage{amsmath,amscd}
\usepackage{pstricks}
\usepackage{pstricks,pst-node}
\usepackage{mathrsfs}
\usepackage{enumitem}
\usepackage[all]{xy}

\newcommand{\ttE}{\mathsf{E}}
\newcommand{\ttF}{\mathsf{F}}
\newcommand{\ttK}{\mathsf{K}}
\newcommand{\R}{\mathsf{R}}

\newcommand{\epmu}{\varepsilon^\mu}
\newcommand{\tepmu}{\widetilde\varepsilon^\mu}
\newcommand{\epvmu}{\varepsilon_v^\mu}
\newcommand{\tepvmu}{\widetilde\varepsilon_v^\mu}
\newcommand{\epvla}{\varepsilon_v^\lambda}
\newcommand{\tepvla}{\widetilde\varepsilon_v^\lambda}
\newcommand{\epla}{\varepsilon^\lambda}

\newcommand{\Phimn}{\Phi_{m|n}}
\newcommand{\Phimnv}{\Phi_{m|n}}
\newcommand{\C}{\mathbb{C}}
\newcommand{\Sm}{\mathfrak{S}_m}
\newcommand{\Sn}{\mathfrak{S}_n}
\newcommand{\Topx}{\operatorname{Top}_x}
\newcommand{\degx}{\deg_x}
\newcommand{\ellp}{\ell}
\newcommand{\len}{\ell}

\newcommand{\ud}{\unlhd}

\DeclareMathAlphabet{\mathpzc}{OT1}{pzc}{m}{it}

\theoremstyle{plain}
\newtheorem{Thm}{Theorem}[section]
\newtheorem{Prop}[Thm]{Proposition}
\newtheorem{Lem}[Thm]{Lemma}
\newtheorem{Coro}[Thm]{Corollary}
\theoremstyle{definition}

\newtheorem{Def}[Thm]{Definition}

\newtheorem{Rem}[Thm]{Remark}

\numberwithin{equation}{section}

\def\fS{{\frak S}}

\newcommand{\msZ}{\mathscr Z}

\newcommand{\msP}{\mathscr P}
\newcommand{\msPr}{\mathscr P_r}
\newcommand{\msPmn}{\mathscr P^{m,n}}

\newcommand{\msN}{\mathscr N}

\newcommand{\msD}{\mathscr D}

\def\sfe{{\mathsf e}}
\def\sff{{\mathsf f}}

\def\sfk{{\mathsf k}}

\def\sH{{\mathcal H}}

\def\sQ{{\mathbb C(v)}}

\def\sX{{\mathcal X}}

\newcommand{\bfHr}{{\mathcal H}_v(\frak{S}_r)}

\newcommand{\bfSr}{{{\mathcal S}}_v(m|n,r)}

\newcommand{\mbnn}{\mathbb N^{n}}
\newcommand{\mbnm}{\mathbb N^{m}}
\newcommand{\mbn}{\mathbb N}

\newcommand{\mbc}{\mathbb C}
\newcommand{\mbz}{\mathbb Z}

\newcommand{\tts}{ {s}}

\newcommand{\End}{\operatorname{End}}
\newcommand{\Hom}{\operatorname{Hom}}

\newcommand{\la}{{\lambda}}
\newcommand{\La}{\Lambda}

\newcommand{\Ga}{{\Lambda}}
\newcommand{\Gam}{{\Lambda}_m}

\newcommand{\dt}{\delta}

\newcommand{\up}{v}

\newcommand{\al}{\alpha}
\newcommand{\bt}{\beta}
\newcommand{\sg}{\sigma}

\newcommand{\bsH}{\sH}

\def\ggp#1#2{\left[\kern-3.2pt\left[{#1\atop #2}\right]\kern-3.2pt\right]}

\def\leq{\leqslant}\def\geq{\geqslant}
\def\le{\leqslant}\def\ge{\geqslant}

\newcommand{\ot}{\otimes}

\newcommand{\bin}{\bigcup}
\newcommand{\jiao}{\bigcap}
\newcommand{\han}{\subseteq}

\newcommand{\ti}{\widetilde}

\newcommand{\ra}{\rightarrow}

\newcommand{\bfdim}{{\mathbf{dim\,}}}

\newcommand{\Par}{\msP_r}
\newcommand{\Pamr}{\mathbb{P}(m,r)}
\newcommand{\Hmn}{{H}(m|n)}
\newcommand{\Hmnr}{{H}(m|n,r)}
\newcommand{\Hpowmnr}{{H}^{\vee}(m|n,r)}
\newcommand{\Hpow}{H^{\vee}}
\newcommand{\Hnmr}{{H}(n|m,r)}

\newcommand{\fSr}{\fS_r}
\newcommand{\ZHr}{\mathscr{Z}\left(\bfHr\right)}
\newcommand{\ZSq}{\mathscr{Z}\left(\Sqmnr\right)}

\def \pa#1{\bar{#1}}

\newcommand{\glmn}{\mathfrak{gl}_{m|n}}

\newcommand{\bU}{{U}_v(\frak{gl}_{m|n})}
\newcommand{\Uqglmn}{{U}_v(\frak{gl}_{m|n})}

\newcommand{\Sqmnr}{\mathcal S_v(m|n, r)}
\newcommand{\Sqmnrz}{\mathcal S_v^0(m|n, r)}
\newcommand{\Uglmn}{\mathcal{U}(\mathfrak{g l}_{m|n})}
\newcommand{\Smnr}{\mathcal S(m|n, r)}
\newcommand{\Lav}{\Lambda_v}

\newcommand{\Lamnrv}{\Lambda^r_{m|n,v}}
\newcommand{\Lamnv}{\Lambda_{m|n,v}}
\newcommand{\Lar}{\Lambda^r}

\newcommand{\Lamn}{\Lambda_{m|n}}
\newcommand{\Lamnr}{\Lambda^r_{m|n}}
\newcommand{\LambdaT}{{\Lambda}}
\newcommand{\tiLamn}{{\Lambda}_{m|n}}
\newcommand{\tiLamnr}{{\Lambda}^r_{m|n}}

\newcommand{\hs}{\mathrm{hs}}
\newcommand{\hm}{\mathrm{hm}}
\newcommand{\hp}{\mathrm{hp}}
\newcommand{\hpk}{\mathrm{hp}_k(x_1,\dots,x_m;y_1,\dots,y_n)}
\newcommand{\hpla}{\mathrm{hp}_\lambda(x_1,\dots,x_m;y_1,\dots,y_n)}
\newcommand{\zr}{\zeta_r}
\newcommand{\zrcl}{\zeta_r^{\mathrm{cl}}}
\newcommand{\etarcl}{\eta_r^{\mathrm{cl}}}
\newcommand{\Zr}{\zeta_r}
\newcommand{\Er}{\eta_r}
\newcommand{\er}{\xi_r}

\newcommand{\hb}[1]{\bar{#1}}
\newcommand{\hi}{{\bar{i}}}
\newcommand{\os}[1]{\bar{#1}}

\begin{document}
\title{Centers of quantum Schur superalgebras from Hecke algebras}

\author{Qiang Fu}
\address{School of Mathematical Sciences,
Key Laboratory of Intelligent Computing and Applications (Ministry of Education),
Tongji University, Shanghai, 200092, China.}
\email{q.fu@hotmail.com, q.fu@tongji.edu.cn}

\author{Yingshan Luo}
\address{School of Mathematical Sciences, Tongji University, Shanghai, 200092, China.}
\email{luoyingshan430@163.com}

\author{Chengquan Sun}
\address{School of Mathematics and Statistics, Ningbo University, Ningbo, 315211, China}
\email{sunchengquan@nbu.edu.cn}

\thanks{Supported by the National Natural Science Foundation
of China (12371032, 12431002)}

\begin{abstract}
We study the center of the quantum Schur superalgebra $\mathcal{S}_v(m|n,r)$ associated with the general linear Lie superalgebra $\mathfrak{gl}_{m|n}$.
Using the super Schur--Weyl duality due to Mitsuhashi between the quantum supergroup $U_v(\mathfrak{gl}_{m|n})$ and the Hecke algebra $\mathcal{H}_v(\mathfrak{S}_r)$,
we transfer two known bases of the center of $\mathcal{H}_v(\mathfrak{S}_r)$, namely the Geck--Rouquier basis and the Jones basis, to the center of $\mathcal{S}_v(m|n,r)$.
This yields two distinct bases for $\mathscr{Z}(\mathcal{S}_v(m|n,r))$, indexed respectively by the symmetrized hook set
$H^{\vee}(m|n,r):=H(\min(m,n)\mid \max(m,n),r)$ and by the full $(m|n)$-hook set $H(m|n,r)$.
Our approach relies on a detailed analysis of the ring $\Lambda_{m|n}$ of doubly symmetric polynomials satisfying $f|_{x_m=t=-y_n}$ independent of $t$, and of its power-sum bases.
\end{abstract}

\sloppy \maketitle
\section{Introduction}
The study of Lie superalgebras as a natural extension of Lie algebras was originally motivated by supersymmetry in mathematical physics.
As the supersymmetric counterparts of quantum groups, quantum supergroups are intimately connected with supersymmetric integrable lattice models and conformal field theories.
The centers of algebras arising from quantum groups and Hecke algebras have long been a central theme, and the center of a quantum superalgebra is of particular interest for its mathematical depth and direct physical relevance.
The centers of quantum supergroups were investigated in \cite{BY,LWY}, while the  center of the Hecke algebra $\mathcal{H}_v(\mathfrak{S}_r)$ has received extensive attention  (see, e.g., \cite{GR,Jones,WW}). Notably,  Geck--Rouquier \cite{GR} and Jones \cite{Jones} constructed explicit bases of $\mathscr{Z}(\mathcal{H}_v(\mathfrak{S}_r))$, and Wan--Wang \cite{WW} established a linear isomorphism between $\mathscr{Z}(\mathcal{H}_v(\mathfrak{S}_r))$ and the degree-$r$ homogeneous component $\Lambda_v^r$ of the ring of symmetric functions $\Lambda_v$.
In the classical limit $v=1$, Geetha and Prasad~\cite{GePr} gave a basis of the center of the Schur algebra arising from conjugacy classes in the symmetric group. Subsequently, Fu~\cite{Fu21} employed the results of~\cite{GR,Jones,WW} to construct bases for the center of the quantum Schur algebra $\mathcal{S}_v(n,r)$.

In this paper, we extend these results to the super setting.
However, the extension is not entirely straightforward. The proof of Theorem~\ref{thm:corrected-power-product}  is by no means a routine transfer from the non-super case: while the corresponding power-sum basis is straightforward in the non-super setting, in the super setting the naively expected indexing set of $(m|n)$-hook partitions of $r$, $H(m|n,r)$, fails when $m>n$; the correct indexing set is $H^\vee(m|n,r):=H(a|b,r)$, where $a=\min(m,n)$ and $b=\max(m,n)$. To prove this, we first handle the case $m\le n$, which demands a delicate analysis of the highest $x$-degree terms for the degree-$r$ homogeneous component $\Lambda_{m|n}^r$ of the ring $\Lambda_{m|n}$ of doubly symmetric polynomials satisfying $f|_{x_m=t=-y_n}$ independent of $t$: we prove the linear independence of a carefully chosen family of vectors, and then, since the dimension of $\Lambda_{m|n}^r$ is known to be the number of $(m|n)$-hook partitions of $r$, we conclude that this family forms a basis. The case $m>n$ is then obtained from the case $m\le n$ via the linear isomorphism $\tau:\Lambda_{n|m}^r\cong \Lambda_{m|n}^r$.
The subtlety of the indexing issue is illustrated by the explicit dependence example in Remark~\ref{Rem}, where using $H(m|n,r)$ with $m>n$ fails.

Let $U_v(\mathfrak{gl}_{m|n})$ be the quantum supergroup associated with the general linear Lie superalgebra $\mathfrak{gl}_{m|n}$, and let $\mathcal{S}_v(m|n,r)$ be the quantum Schur superalgebra, which is the image of $U_v(\mathfrak{gl}_{m|n})$ in the endomorphism algebra of the $r$-fold tensor power of the natural representation.
The super Schur--Weyl duality established by Mitsuhashi \cite{Mi} provides a double centraliser property between $U_v(\mathfrak{gl}_{m|n})$ and the Hecke algebra $\mathcal{H}_v(\mathfrak{S}_r)$ acting on the same tensor space.
This duality allows us to relate the center of the quantum Schur superalgebra to the center of the Hecke algebra.

Our main contribution is the construction of two explicit bases for the center $\mathscr{Z}(\mathcal{S}_v(m|n,r))$.
Using the isomorphism between $\mathscr{Z}(\mathcal{H}_v(\mathfrak{S}_r))$ and $\Lambda_v^r$ due to Wan and Wang \cite{WW},  we show that,
when restricted to the appropriate indexing sets, the images under the duality map $\eta_r$ of the Geck--Rouquier basis $\{f_\lambda^*\}$ and the Jones basis $\{\mathcal{N}_\lambda(F_\lambda)\}$ form bases of $\mathscr{Z}(\mathcal{S}_v(m|n,r))$.
The index sets are:
\begin{itemize}[nosep]
\item for the Jones basis:
\(H^{\vee}(m|n,r)\);
\item for the Geck--Rouquier basis, we give two variants, namely one indexed by \(H^{\vee}(m|n,r)\) obtained via integral forms and specialization at \(v=1\), and another indexed by the full hook set \(H(m|n,r)\) obtained via a power-series embedding.
\end{itemize}
The proof of the basis property for the Jones basis relies on the power-sum basis of $\Lambda_{m|n}^r$ established in Theorem~\ref{thm:corrected-power-product}.
For the Geck--Rouquier basis, we use two different techniques: one uses an integral form of the Hecke algebra and reduces to the classical symmetric group case, the other uses a formal power-series argument that avoids integral forms and yields the full hook set as index set.

Kemp and Ramgoolam  \cite{KempRamgoolam2020} showed that, via Schur--Weyl duality, the problem of distinguishing half-BPS states in \(\mathcal{N}=4\) SYM translates into a structural problem about the center of the symmetric group algebra. This suggests that the study of centers of algebras arising from Schur--Weyl duality is of physical interest. The explicit bases for the center of the quantum Schur superalgebra established in this paper may provide algebraic data relevant to supersymmetric physical settings.

The paper is organized as follows.
In \S2 we recall the definition of the quantum supergroup $U_v(\mathfrak{gl}_{m|n})$, the natural module and its tensor powers, and the quantum Schur superalgebra, together with the super Schur--Weyl duality due to Mitsuhashi.
\S3 is devoted to the doubly symmetric polynomial subring $\Lambda_{m|n}$.
We establish power-sum bases of its homogeneous components $\Lambda_{m|n}^r$ (Theorem~\ref{thm:corrected-power-product}),  which are crucial for the later transfer arguments.
In \S4, we prove that the center of $\mathcal{S}_v(m|n,r)$ consists only of even elements, so the ordinary center equals the supercenter.
Then we recall the Geck--Rouquier and Jones bases of $\mathscr{Z}(\mathcal{H}_v(\mathfrak{S}_r))$ and the isomorphism between this center and $\Lambda_v^r$ due to Wan and Wang.
We then transfer these bases to the quantum Schur superalgebra, obtaining our main theorems (Theorems~ \ref{Jones basis}, \ref{basis of Sqmnr} and \ref{GR basis of quantum Schur superalgebra}).
Finally, we specialise to $v=1$ and describe bases for the center of the classical Schur superalgebra (Theorem~\ref{classical case}).

\section{Quantum supergroup and quantum Schur superalgebra}
In this section we introduce the quantum supergroup $U_v(\mathfrak{gl}_{m|n})$, its natural module and tensor representations, and define the quantum Schur superalgebra $\mathcal{S}_v(m|n,r)$ as the image of $U_v(\mathfrak{gl}_{m|n})$ in $\operatorname{End}(V_v^{\otimes r})$. We then recall the super Schur--Weyl duality, which will be the main tool for relating the center of the quantum Schur superalgebra to that of the Hecke algebra.
\subsection{The quantum supergroup $\bU$}
We begin by fixing the parity convention. For $m,n\in\mbn$  define the map
$\hb{\ } : \{ 1, 2, \dots ,m+n \} \ra \mbz_2$ by setting
$$ \hi =
\begin{cases}
\os{0} & \mbox{if} \ 1 \leq i \leq m,
\\ \os{1} & \mbox{if} \ m+1 \leq i \leq m+n.
\end{cases}   $$
Let $\sQ$ be the field of rational functions in an indeterminate $v$. For $1\leq i\leq m+n$ let
$\up_{i}=\up^{(-1)^{\hi}}$.

Following \cite{ZR},
we recall the definition of the quantum superalgebra $\bU$ over $\sQ$.
We also recall the definition of the super commutator on homogeneous elements $X,Y$ of a superalgebra with parity function $\hb{ \ }$ given by
$ [X,Y]=XY-(-1)^{\hb{X} \hb{Y}}YX .$

\begin{Def}\label{Def. of U}
The quantum superalgebra $\bU$ over $\sQ$  is the algebra generated by
 the elements
$$\ttE_{i,i+1},\ \ttE_{i+1,i} , \ \ttK_j,\ \ttK_j^{-1}\ (1\leq i\leq m+n-1, \ 1\leq j\leq m+n)$$
which satisfy the following relations:
\begin{itemize}
\item[(a)] $\ttK_{i}\ttK_{j}=\ttK_{j}\ttK_{i},\ \ttK_{i}\ttK_{i}^{-1}=1;$
\item[(b)] $ \ttK_{i}\ttE_{j,j+1}=\up_{i}^{\dt_{i,j}-\dt_{i,j+1}} \ttE_{j,j+1}\ttK_{i},$
$\ttK_{i}\ttE_{j+1,j}=\up_{i}^{-\dt_{i,j}+\dt_{i,j+1}}
 \ttE_{j+1,j}\ttK_{i};$
\item[(c)] $\ttE_{i,i+1}\ttE_{j,j+1}=\ttE_{j,j+1}\ttE_{i,i+1},
\ \ttE_{i+1,i} \ttE_{j+1,j}=\ttE_{j+1,j} E_{i+1,i}  \ \mbox{when} \ |i-j|>1;$
\item[(d)]
$[\ttE_{i,i+1},\ttE_{j+1,j}]=\dt_{ij}\frac
{\ttK_{i}\ttK_{i+1}^{-1}-\ttK_{i}^{-1}\ttK_{i+1}}{\up_i-\up_i^{-1}};$
\item[(e)] for $i \neq m$ and $|i-j|=1$,
\begin{equation*}
\begin{split}
& \ttE_{i,i+1}^2\ttE_{j,j+1}-(\up_i+\up_i^{-1})\ttE_{i,i+1}\ttE_{j,j+1}\ttE_{i,i+1}+
\ttE_{j,j+1}\ttE_{i,i+1}^2=0,\\
& \ttE_{i+1,i}^2\ttE_{j+1,j}-(\up_i+\up_i^{-1})\ttE_{i+1,i}\ttE_{j+1,j}\ttE_{i+1,i}+
\ttE_{j+1,j}\ttE_{i+1,i}^2=0;
\end{split}
\end{equation*}
\item[(f)]
$\ttE_{m,m+1}^2=\ttE_{m+1,m}^2=0$ and $[\ttE_{m,m+1},\ttE_{m-1,m+2}]=[\ttE_{m+1,m},\ttE_{m+2,m-1}]=0$.
\end{itemize}
Here, only  $\ttE_{m,m+1}$ and $\ttE_{m+1,m}$ are odd generators and others are even generators.
The quantum root vectors $\ttE_{i,j}$ are defined recursively as follows:
$$\ttE_{i,j} =
\begin{cases}
\ttE_{i,k}\ttE_{k,j}-\up_k \ttE_{k,j}\ttE_{i,k} & \mbox{if} \ i>j, \\
\ttE_{i,k}\ttE_{k,j}-\up_k^{-1} \ttE_{k,j}\ttE_{i,k} & \mbox{if} \ i<j,
\end{cases} $$
where $k$ can be taken to be an arbitrary index strictly between $i$ and $j$.
\end{Def}

\subsection{The natural module and tensor representations}
Let $V_\up$ be the vector superspace over $\mbc(\up)$ with dimension vector $\bfdim(V_\up)=(m|n)$. Fix a basis $v_1, \dots, v_{m+n}$, and define the $\mbz_2$-grading by $\pa v_a=\pa a$. The space  $V_\up$ admits a natural $\Uqglmn$-module structure,
 given on the basis by
\begin{equation}\label{Z}
 \ttK_a v_b  =\up_a^{\delta_{a, b}} v_b, \quad
 \ttE_a v_b  =\delta_{a+1, b} v_a, \quad
 \ttF_a v_b   =\delta_{a, b} v_{a+1},
\end{equation}
where $\ttE_a=\ttE_{a,a+1}$ and $\ttF_a=\ttE_{a+1,a}$.

To obtain representations on tensor powers, we equip $\Uqglmn$ with a comultiplication
defined on generators by
\begin{equation}\label{comul}
\begin{aligned}
\Delta\left(\ttE_a\right) & =\ttE_a \otimes \ttK_a^{-1} \ttK_{a+1}+1 \otimes \ttE_a, \\
\Delta\left(\ttF_a\right) & =\ttF_a \otimes 1+\ttK_a \ttK_{a+1}^{-1} \otimes \ttF_a, \\
\Delta\left(\ttK_a\right) & =\ttK_a \otimes \ttK_a .
\end{aligned}
\end{equation}
Via this comultiplication, for any $r\geq 1$, the $r$-fold tensor product of the natural module,
$
V_\up^{\otimes r}:=V_\up \otimes V_\up \otimes \dots \otimes V_\up
$
admits an action of $\Uqglmn$. Consequently, we obtain a  superalgebra homomorphism:
\begin{equation}\label{Zr}
\Zr: \Uqglmn \rightarrow \End_{\mbc(\up)} (V_\up^{\otimes r}).
\end{equation}
The image
$$
\Sqmnr=\Zr(\Uqglmn)
$$
is called the quantum Schur superalgebra (see \cite{EK,Mi}).

\subsection{Hecke algebras and super Schur--Weyl duality}
The  Hecke algebra $\bfHr$
associated with the symmetric group $\fSr$  is  the $\mbc(\up)$-algebra
generated by
$T_i$ ($1\leq i\leq r-1$),
subject to the following
relations:
\begin{equation}\label{Hecke algebra}
\aligned
 & (T_i+1)(T_i-q)=0,\;\;
 T_iT_{i+1}T_i=T_{i+1}T_iT_{i+1},\;\;T_iT_j=T_jT_i\;(|i-j|>1),
 \endaligned
\end{equation}
where $q=v^2$.

For each \(w\in\mathfrak S_r\), take any reduced expression \(w=s_{i_1}\cdots s_{i_\ell}\) with \(s_i=(i,i+1)\), and set \(T_w:=T_{i_1}\cdots T_{i_\ell}\); the defining relations of the Hecke algebra ensure this is well-defined. Then \(\{T_w\mid w\in\mathfrak S_r\}\) is the standard \(\mbc(v)\)-basis of \(\mathcal H_v(\mathfrak S_r)\).

There is a natural signed action of $\bfHr$ on $V_\up^{\otimes r}$, which gives a homomorphism
$$
\Er: \bfHr^{\operatorname{op}} \rightarrow \End_{\mbc(\up)} (V_\up^{\otimes r}),
$$
(see \cite{Mi}).

The following fundamental result, due to Mitsuhashi \cite{Mi}, establishes the Schur--Weyl duality between these two actions.
\begin{Thm}\label{doublecentralizer}
The actions of $\Uqglmn$ and $\bfHr$ on $V_\up^{\otimes r}$ commute. Moreover, we have
 $\Sqmnr = \End_{\bfHr} (V_\up^{\otimes r})$ and $\Er(\bfHr^{\operatorname{op}}) = \End_{\Uqglmn} (V_\up^{\otimes r})$.
\end{Thm}

For a partition $\lambda$ of $r$, we denote by $S_\up^\lambda$ the Specht module of $\bfHr$. It is well known that $\left\{S_\up^\lambda \mid \lambda \in \Par \right\}$ forms a complete list of  pairwise non-isomorphic finite dimensional irreducible $\bfHr$-modules.
According to \cite[Th. 5.1]{Mi} and \cite[Th. 3.3]{Mi1}, we have the following $\Uqglmn$-$\bfHr$-bimodule decomposition:
\begin{equation}\label{Qdecom}
  V_\up^{\otimes r}=\bigoplus_{\la \in \Hmnr} L_\up(\la)\otimes S_\up ^\la,
\end{equation}
where $L_\up(\la)=\Hom_{\bfHr}(S_\up ^\la,V_v^{\ot r})$ $(\la \in \Hmnr)$ are precisely all irreducible non-isomorphic $\bfSr$-modules. This decomposition will be crucial in the subsequent analysis of the center.

\section{ Power-sum bases of the ring $\tiLamn$}
In order to transfer the known bases of the Hecke algebra center to the quantum Schur superalgebra, we need a detailed understanding of the subring $\Lambda_{m|n}$ of doubly symmetric polynomials satisfying $f|_{x_m=t=-y_n}$ independent of $t$, and of its power-sum bases.
\subsection{The ring $\tiLamn$}
Let $\msPr$ denote the set of partitions of $r$, and $\msP=\bin_{r\geq 0}\msPr$.
A partition $\la=(\la_1, \la_2, \dots)$ is called an $(m|n)$-hook partition if $\la_{m+1} \leq n$. We denote by $\Hmn$ the set of all $(m | n)$-hook partitions, and set $\Hmnr=\msPr\jiao \Hmn$.

Let $\Ga$ be the ring of symmetric functions in countably many independent variables $x=\{x_1,x_2,\dots\}$ over $\mbc$.
For $r\geq 0$ let  $\Ga^r$ be the subspace of $\Ga$ consisting of the homogeneous symmetric polynomials of degree $r$, together with the zero polynomial.
For $k\geq 1$ let
$$p_k(x)=\sum_{i\geq 1}x_i^k\in\Ga^k.$$
For $\la\in\msPr$ let
\begin{equation}\label{pla}
 p_\la(x)=\prod_{k\geq 1}p_{\la_k}(x)\in\Ga^r.
\end{equation}

Let $\left\{x_1, x_2, \dots, x_m\right\}$ and $\left\{y_1, y_2, \dots, y_n\right\}$ be
two sets of independent indeterminates. A polynomial $f$ in $\mbc[x_1, x_2, \dots, x_m, y_1, y_2, \dots, y_n]$ is called doubly symmetric if the following conditions are satisfied:
\begin{itemize}
\item[(1)] $f$ is symmetric in $x_1, \dots, x_m$.
\item[(2)] $f$ is symmetric in $y_1, \dots, y_n$.
\end{itemize}
Define $\tiLamn$ to be the subring of the ring of doubly symmetric polynomials over $\mbc$ in the two sets of variables $\left\{x_1, \dots, x_m\right\}$ and $\left\{y_1, \dots, y_n\right\}$ consisting of polynomials $f$ such that the specialization $\left.f\right|_{x_m=t=-y_n}$ is independent of $t$.
There is a ring epimorphism
\begin{equation}\label{Phimn}
 \Phimn:\La\ra\tiLamn
\end{equation}
given on power sums by
$$
 \Phimn (p_r(x) )=p_r(x_1,\dots,x_m)+(-1)^{r-1}p_r(y_1,\dots,y_n)
$$ where $p_r(x_1,\dots,x_m)=\sum_{i=1}^m x_i^r$ and $p_r(y_1,\dots,y_n)=\sum_{j=1}^n y_j^r
$
(see \cite[A.2.2]{CW}).

For $\la\in\msPr$, let $\tts_\la(x)\in\Ga^r$ be the corresponding Schur function.
For $\la\in\Hmnr$ let $\hs_\lambda(x_1,\dots,x_m ; y_1,\dots,y_n)$ be the corresponding super Schur function.
For $r\geq 0$ let  $\tiLamnr$ be the subspace of $\tiLamn$ consisting of the homogeneous polynomials of degree $r$, together with the zero polynomial.

The following result can be found in \cite[Lemma 6.4]{BR}, \cite{Se} and \cite[A.2.2]{CW}.
\begin{Thm}\label{super Schur function}
For $\la\in\msPr$ we have
$$\Phimn(\tts_\lambda(x))=
\begin{cases}
\hs_\la(x_1,\dots,x_m ; y_1,\dots,y_n)& \text{if $\la\in\Hmnr$}\\
0&\text{otherwise}.
\end{cases}
$$
Furthermore the set
$\{\hs_\lambda(x_1,x_2,\dots,x_m ; y_1,y_2,\dots,y_n) \mid \la \in \Hmnr \}$ forms a $\mbc$-basis of $\tiLamnr$.
\end{Thm}

For $\la\in\msPr$ let $m_\la(x)\in\Ga^r$ be the monomial symmetric function, and define
\begin{equation}\label{hmla}
 \hm_\la(x_1,\dots,x_m ; y_1,\dots,y_n)=\Phimn(m_\la(x)).
\end{equation}
Let $\ud$ denote the dominance order on $\msPr$.
\begin{Lem}\label{hm_lambda}
The set
$\{\hm_\lambda(x_1,\dots,x_m ; y_1,\dots,y_n) \mid \la \in \Hmnr \}$ forms a $\mbc$-basis for $\tiLamnr$.
\end{Lem}
\begin{proof}
By \cite[I,(6.5)]{Macd}, for $\la\in\msPr$ we have
\begin{equation}\label{s_lam_mu}
  \tts_\la(x)=\sum_{\mu\ud \la,\,\mu\in \Par }K_{\la,\mu}m_\mu(x),
\end{equation}
where $K_{\la,\mu}$ are Kostka numbers. Consequently,
\begin{equation*}
 m_\la(x)=\sum_{\mu\ud \la,\,\mu\in \Par }f_{\la,\mu}\tts_\mu(x),
\end{equation*}
for some $f_{\la,\mu}\in\mbz$.  Applying $\Phimn$ and using Theorem \ref{super Schur function}, we obtain for $\la\in\Hmnr$
\begin{equation*}
 \hm_\la(x_1,\dots,x_m ; y_1,\dots,y_n)=\sum_{\mu\ud \la,\,\mu\in \Hmnr }f_{\la,\mu}\hs_\mu(x_1,\dots,x_m ; y_1,\dots,y_n).
\end{equation*}
The transition matrix from the super Schur functions to the super monomial symmetric functions is unitriangular with respect to the dominance order.
Therefore the set $\{\hm_\lambda(x_1,\dots,x_m ; y_1,\dots,y_n) \mid \la \in \Hmnr \}$ forms a $\mbc$-basis of $\tiLamnr$.
\end{proof}

 \subsection{The doubly symmetric polynomial ring $R$}
For $k\geq 1$ let
\[
        \hpk=\Phimn(p_k(x))=\sum_{1\leq i\leq m}x_i^r+(-1)^{r-1}\sum_{1\leq j\leq n}y_j^r
\]
where   $\Phimn$ is defined in \eqref{Phimn}. For a partition $\lambda\in\msP$,
define
\[
        \hpla=\prod_{k\geq 1} \hp_{\lambda_k}(x_1,\dots,x_m;y_1,\dots,y_n).
        \]
For $1\le j\le n$
set
\[
        z_j=-y_j.
\]
Then for  $k\geq 1$
$$\sigma_k:=\hpk=p_k(x_1,\dots,x_m)-p_k(z_1,\dots,z_n)$$
and for $\la\in\msPr$
\[
 \sigma_\lambda:=\prod_{k\geq 1}\sg_{\la_k}=\hpla.
\]

Let $$\Ga_m=\mbc[x_1,x_2,\dots,x_m]^{\frak{S}_m}$$ be the ring of symmetric functions in $m$ independent variables $x_1,x_2,\dots,x_m$.
Then $$\Ga_m=\bigoplus_{r\geq 0}\Ga_m^r,$$ where
$\Ga_m^r$ consists of the homogeneous symmetric polynomials of degree $r$, together with the zero polynomial.

Let
\[
        R=\C[x_1,\dots,x_m,z_1,\dots,z_n]^{\Sm\times\Sn}
\]
be the doubly symmetric polynomial ring.  For $1\le i\le m$ and
$1\le j\le n$ let
\[X_i=p_i(x_1,\dots,x_m),\quad Z_j=p_j(z_1,\dots,z_n).
\]
Then we have
\[
        R=\C[X_1,\dots,X_m,Z_1,\dots,Z_n]=\C[X_1,\dots,X_m,\sigma_1,\dots,\sigma_n].
\]
Clearly we have the following result.
\begin{Lem} \label{lem:alg-independence}
The elements $X_1,\dots,X_m,\sigma_1,\dots,\sigma_n$ are algebraically
independent.  In particular,
$$B:=\C[\sigma_1,\dots,\sigma_n]$$
is a polynomial ring, and
$R=B\otimes_\C \Ga_m.$
\end{Lem}

For $f\in R$, write uniquely
\[
        f=\sum_{r\geq 0} f_r,
        \qquad
        f_r\in B\otimes_\C \Ga^r_m.
\]
For $f\not=0\in R$ define
\[
        \degx(f)=\max\{r\mid f_r\ne 0\},
        \qquad
        \Topx(f)=f_{\degx(f)}.
\]

\subsection{Highest $x$-degree term of the $\sg_\al$}
Let
\[
        E_x(t)=\prod_{i=1}^m(1-x_i t),
        \qquad
        E_z(t)=\prod_{j=1}^n(1-z_j t).
\]
Then
\[
        -\ln E_x(t)=\sum_{a\ge 1}\frac{p_a(x_1,\dots,x_m)}{a}t^a,
        \qquad
        -\ln E_z(t)=\sum_{a\ge 1}\frac{p_a(z_1,\dots,z_n)}{a}t^a.
\]
Therefore
\begin{equation}\label{ln frac E_z(t)E_x(t)}
\ln\frac{E_z(t)}{E_x(t)}
=\sum_{a\ge 1}\frac{\sigma_a}{a}t^a.
\end{equation}
Define
\[
        C(t)=\exp\left(\sum_{i=1}^n\frac{\sigma_i t^i}{i}\right)
        =\sum_{j\ge 0}c_jt^j,
        \qquad c_j\in B.
\]

\begin{Lem}\label{lem:congruence}
We have
\[
        E_x(t)C(t)\equiv E_z(t)\pmod{t^{n+1}}.
\]
Equivalently,
\[
        E_x(t)C(t)=E_z(t)+T(t),
        \qquad
        T(t)=\sum_{r\ge n+1}T_rt^r.
\]
\end{Lem}
\begin{proof}
By \eqref{ln frac E_z(t)E_x(t)} we have
\begin{align*}
        \ln\frac{E_x(t)C(t)}{E_z(t)}
        &=-\ln \frac{E_z(t)}{E_x(t)}+\ln C(t)
         =-\sum_{a\ge 1}\frac{\sigma_a}{a}t^a+\sum_{1\le a\le n}\frac{\sigma_a}{a}t^a=
         -\sum_{a\ge n+1}\frac{\sigma_a}{a}t^a
\end{align*} Hence
$\ln\frac{E_x(t)C(t)}{E_z(t)}\in t^{n+1}R[[t]].$
Since the exponential of a series in $t^{n+1}R[[t]]$ is congruent to $1$ modulo
$t^{n+1}$, the desired congruence follows.
\end{proof}

We use the standard convention that \([t^a]f(t)\) denotes the coefficient of \(t^a\) in a formal power series \(f(t)\).
For $r\geq 0$ let $h_r(x_1,\dots, x_m)$ be the complete symmetric function in
$x_1,\dots,x_m$ and $e_r(x_1,\dots, x_m)$ be the elementary symmetric function in
$x_1,\dots,x_m$.
\begin{Lem}\label{lem:top-one}
 Assume $n\geq m$.  Let
\begin{equation*}\label{the number d}
        d=n-m+1.
\end{equation*}
Then  $c_d\not=0$. Furthermore for $a>n$, we have
$$ \Topx(\sigma_a)
        =a(-1)^{m+1}c_d e_m(x_1,\dots, x_m)h_{a-n-1}(x_1,\dots, x_m).$$
\end{Lem}

\begin{proof}
Since $n\ge m\ge 1$, we have $1\le d\le n$.
Expanding the exponential gives
\[
        C(t)=1+\sum_{i=1}^n\frac{\sigma_i t^i}{i}
        +\frac{1}{2!}\left(\sum_{i=1}^n\frac{\sigma_i t^i}{i}\right)^2+\dots.
\]
Therefore
\[
        c_d= \frac{\sigma_d}{d}
        +\text{terms of ordinary degree at least }2
        \text{ in }\sigma_1,\dots,\sigma_n.
\]
By Lemma~\ref{lem:alg-independence}, $B=\C[\sigma_1,\dots,\sigma_n]$ is a
polynomial ring.  The linear term $\sigma_d/d$ cannot be cancelled by terms of
ordinary degree at least two.  Hence $c_d\ne 0$.

By Lemma~\ref{lem:congruence},
\[
        \frac{E_z(t)}{E_x(t)}
        =C(t)-\frac{T(t)}{E_x(t)}
        =C(t) (1-U(t) ),
\]
where
\begin{equation}\label{U(t)}
 U(t)=\frac{T(t)}{E_x(t)C(t)}.
\end{equation}
Thus by \eqref{ln frac E_z(t)E_x(t)}
\[
        \sum_{i\ge 1}\frac{\sigma_i}{i}t^i=\ln\frac{E_z(t)}{E_x(t)}=\ln C(t)+\ln(1-U(t))=\sum_{i=1}^n\frac{\sigma_i t^i}{i}+\ln(1-U(t)).
\]
Hence
\[\sum_{i\ge n+1}\frac{\sigma_i}{i}t^i=
\ln(1-U(t))=-U(t)-\frac{U(t)^2}{2}-\frac{U(t)^3}{3}-\dots
\]
Therefore, for $a>n$
\begin{equation}\label{frac sigma a a}
        \frac{\sigma_a}{a}=[t^a]\ln(1-U(t))=-[t^a]U(t)-\sum_{k\geq 2}[t^a]\frac{U(t)^k}{k}.
\end{equation}

We first compute the top $x$-degree part of $[t^a]U(t)$.  Since $E_z(t)$ has
degree $n$ in $t$, by Lemma \ref{lem:congruence}
\[
        T_{b}=[t^{b}]E_x(t)C(t)
\]
for $b\geq n+1$.
Now
\[
        E_x(t)=\sum_{r=0}^m(-1)^r e_r(x)t^r,
        \qquad
        C(t)=\sum_{j\ge 0}c_jt^j.
\]
Therefore
\[
        T_{b}=\sum_{r=0}^m(-1)^r e_r(x_1,\dots,x_m)c_{b-r}
\]
for $b\geq n+1$.
Each $c_{b-r}$ belongs to $B$ and has $x$-degree zero.  Hence
\[
  \deg_x(T_b)\le m
\]
for $b\geq n+1$. Since $c_d\not=0$ we have  $ \deg_x(T_{n+1})=m$ and
$$\Topx(T_{n+1})=(-1)^m c_{d}e_m(x_1,\dots,x_m).$$
Moreover,
\[
        \frac{1}{E_x(t)}=\prod_{i=1}^m\frac{1}{1-x_it}
        =\sum_{r\ge 0}h_r(x_1,\dots,x_m)t^r.
\]
Note that $C(t)^{-1}\in B[[t]]$ has constant term $1$ and all its coefficients have
$x$-degree zero.
Thus, by \eqref{U(t)} we have
\begin{equation}\label{Topx([t^a]U(t))}
        \Topx([t^a]U(t))=(-1)^m c_d e_m(x_1,\dots,x_m)h_{a-n-1}(x_1,\dots,x_m),\quad
        \deg_x([t^a]U(t))=a-d.
\end{equation}

 Write $C(t)^{-1}=\sum_{i\geq 0}(C(t)^{-1})_it^i$. Then we have
\[
        U(t)=T(t)E_x(t)^{-1}C(t)^{-1}=
        \sum_{r\geq n+1}
        \big(\sum_{s\geq n+1,\,j,u\geq 0\atop s+j+u=r}T_s h_j(x_1,\dots,x_m)(C(t)^{-1})_u\big)t^r.
\]
Since $\degx(T_s)\le m$, $\degx(h_j)=j$ and $(C(t)^{-1})_u\in B$, we have
\[
   \degx([t^r]U(t))\leq    m+j=m+r-s-u\le m+r-(n+1)=r-d.
\]
Consequently, for $k\geq 2$ we have $\degx([t^a]U(t)^k)\leq a-kd<a-d$, because $d\ge 1$.  Hence by \eqref{frac sigma a a} and \eqref{Topx([t^a]U(t))}, we have
$\Topx ({\sigma_a} )
        =a(-1)^{m+1}c_d e_m(x_1,\dots,x_m)h_{a-n-1}(x_1,\dots,x_m).$
This completes the proof.
\end{proof}

For $\la\in\msP$ let $\len(\la)$ be the length of $\la$ and $h_\la(x_1,\dots,x_m)=\prod_{i\geq 1}h_{\la_i}(x_1,\dots,x_m)$.
Define
\begin{equation}\label{msPmn}
 \msPmn=\{\alpha\in\msP\mid\alpha_i>n,\,\forall i,\,\ell(\al)\leq m\}.
\end{equation}
\begin{Coro}\label{lem:top-product}
Assume $n\geq m$.
Let $\alpha\in\msPmn$. Then
\[
        \Topx(\sigma_\alpha)
        =K_\alpha c_d^{\ellp}e_m(x_1,\dots,x_m)^{\ellp}
        h_{\rho(\alpha)}(x_1,\dots,x_m),
\]
where $\ellp=\ell(\alpha)$, $\rho(\alpha)=(\alpha_1-n-1,\dots,\alpha_{\ellp}-n-1)$ and
$K_\alpha= (-1)^{\ellp(m+1)} \prod_{i=1}^{\ellp}\alpha_i \in\C^\times.$
\end{Coro}
\begin{proof}
Since $\Topx(fg)=\Topx(f)\Topx(g)$
for nonzero $f,g\in R$, applying Lemma~\ref{lem:top-one} to each factor gives the desired formula.
\end{proof}

\subsection{Linear independence of certain power sums}
We need the following purely symmetric-function fact in the $x$-variables.
For $\la\in\msP$ let $s_\la(x_1,\dots,x_m)\in\Gam$ be the corresponding Schur function.

\begin{Lem}\label{lem:symmetric-independent}
Assume $n\geq m$. Then the set
\[
       \sX:= \{e_m(x_1,\dots,x_m)^j h_\rho(x_1,\dots,x_m)\mid 0\le j\le m,\,\rho\in\msP,\, \len(\rho)\le j\}
\]
is linearly independent over $\C$ in $\Gam$.  Consequently,
it is also linearly independent over $B$ in
$R=B\otimes_\C \Gam.$
\end{Lem}

\begin{proof}
It suffices to prove the first assertion; the second follows by expanding the
coefficients in the standard monomial basis of the polynomial ring $B$.

Fix $j$.  By \cite[I,(6.5)]{Macd} for $\rho\in\msPr$ with $\len(\rho)\leq j$
\[
\begin{split}
h_\rho(x_1,\dots,x_m)&= \sum_{\nu\trianglerighteq\rho,\,\nu\in\msPr}
K_{\nu\rho}\tts_\nu(x_1,\dots,x_m)=
\sum_{\nu\trianglerighteq\rho,\,\nu\in\msPr\atop\len(\nu)\leq j}
K_{\nu\rho}\tts_\nu(x_1,\dots,x_m)
\end{split}
\]
where $K_{\nu\rho}$ are Kostka numbers.
 Using
\[
        e_m(x_1,\dots,x_m)^j \tts_\nu(x_1,\dots,x_m)=\tts_{\nu+(j^m)}(x_1,\dots,x_m),
\]
we see that
\[
\begin{split}
e_m(x_1,\dots,x_m)^j h_\rho(x_1,\dots,x_m)&=
\sum_{\nu\trianglerighteq\rho,\,\nu\in\msPr\atop\len(\nu)\leq j}
K_{\nu\rho}\tts_{\nu+(j^m)}(x_1,\dots,x_m)
\end{split}
\]
If $j<m$ and $\len(\rho)\leq j$, then the $m$-th part of $\rho+(j^m)$
is exactly $j$. If $j=m$ and $\len(\rho)\leq j$, then the $m$-th part of $\rho+(j^m)$ is at least $m$.
Hence the set $$\{\tts_{\rho+(j^m)}(x_1,\dots,x_m)\mid
0\le j\le m,\,\rho\in\msP,\, \len(\rho)\le j
\}$$ is linearly independent over $\C$.
Therefore the set
$\sX$
is linearly independent over $\C$.
\end{proof}

The following proposition is a direct consequence of Lemma \ref{lem:symmetric-independent} and will be essential for the proof of the basis theorem in the next subsection.

\begin{Prop}\label{prop:high-independent}
Assume $n\geq m$. Then the set
$\{\sigma_\alpha\mid\alpha\in\msPmn\}$
is linearly independent over $B$ in $R$, where $\msPmn$ is defined in \eqref{msPmn}.
\end{Prop}

\begin{proof}
Suppose
\[
        \sum_{\alpha\in\msPmn} F_\alpha\sigma_\alpha=0,
        \qquad F_\alpha\in B,
\]
with only finitely many nonzero $F_\alpha$.  If $F_\alpha\not=0$ for some $\alpha\in\msPmn$, let
$D$ be the maximum of $\degx(\sigma_\alpha)$ among those $\alpha$ with
$F_\alpha\ne 0$.  Then we have
\[
     \sum_{\alpha\in\msPmn\atop
         \degx(\sigma_\alpha)=D}F_\alpha\Topx(\sigma_\alpha)=0.
\]
It follows from Corollary~\ref{lem:top-product} that
\[
 \sum_{\alpha\in\msPmn \atop
         \degx(\sigma_\alpha)=D}
F_\alpha K_\alpha c_d^{\len(\alpha)}
e_m(x_1,\dots,x_m)^{\len(\alpha)}
h_{\rho(\alpha)}(x_1,\dots,x_m)=0.
\]
Therefore by Lemma~\ref{lem:symmetric-independent} \[
        F_\alpha K_\alpha c_d^{\len(\alpha)}=0
\]
for all $\alpha\in\msPmn$ with $\degx(\sigma_\alpha)=D$.
By Lemma \ref{lem:top-one} we have
$c_d\ne 0$. Furthermore $K_\alpha\ne 0$. Hence we have $F_\alpha=0$, contradicting the choice of $D$.  Therefore all
$F_\alpha=0$.
\end{proof}

\subsection{The power-sum basis of  $\LambdaT^r_{m|n}$}
For $m,n\ge 1$ and $r\ge 0$ let
\begin{equation}\label{Hpow(m|n,r)}
        \Hpow(m|n,r)=H(\min(m,n)\mid \max(m,n),r).
\end{equation}
With the linear independence result at hand, we can now prove the following basis theorem for the power sums.
\begin{Thm}\label{thm:corrected-power-product}
The set
\[
        \{\hpla\mid \lambda\in \Hpow(m|n,r)\}
\]
is a $\C$-basis of $\LambdaT^r_{m|n}$.
\end{Thm}
\begin{proof}
Suppose first that  $n\ge m$. Then $\Hpow(m|n,r)=H(m|n,r)$.
By Theorem \ref{super Schur function}
$\dim_\C\LambdaT^r_{m|n}=  \# H(m|n,r).$ Hence, after the substitution $z_j=-y_j$, it is enough to prove that
$\{\sigma_\lambda\mid \lambda\in H(m|n,r)\}$
is linearly independent.
Assume
\[
        \sum_{\lambda\in H(m|n,r)}c_\lambda\sigma_\lambda=0.
\]
For each $\lambda\in H(m|n,r)$, split its parts into
\[
        \lambda=\alpha\sqcup\beta,
\]
where all parts of $\alpha$ are greater than $n$, and all parts of $\beta$ are at
most $n$.  Since $\lambda\in H(m|n,r)$, we have $\len(\alpha)\le m$; also
$\beta_1\le n$.  Then
$\sigma_\lambda=\sigma_\alpha\sigma_\beta.$
Because $\beta_1\le n$, we have $\sigma_\beta\in B$.  Therefore the relation
becomes
\[
        \sum_{\alpha\in\msP,\, \len(\alpha)\le m\atop \alpha_i>n,\,\forall i}
        \bigg(\sum_{\alpha\sqcup\beta\in\msPr\atop\beta_1\le n}c_{\alpha\sqcup\beta}\sigma_\beta\bigg)
        \sigma_\alpha=0.
\]
Set
\[
        F_\alpha=\sum_{\alpha\sqcup\beta\in\msPr\atop\beta_1\le n}c_{\alpha\sqcup\beta}\sigma_\beta\in B.
\]
By Proposition~\ref{prop:high-independent}, $F_\alpha=0$ for all $\al$.  Since the
monomials $\{\sigma_\beta\mid \beta_1\le n\}$
form the standard monomial basis of the polynomial ring
$B=\C[\sigma_1,\dots,\sigma_n]$, each coefficient $c_{\alpha\sqcup\beta}$ is
zero.  Therefore the set $\{\sigma_\lambda\mid \lambda\in H(m|n,r)\}$ is linearly independent. Hence the set $$\{\hpla\mid \lambda\in H(m|n,r)\}
$$
is a $\C$-basis of $\LambdaT^r_{m|n}$.

Now assume $m>n$. Then $\Hpow(m|n,r)=H(n|m,r)$.
 There is a linear
isomorphism
\[
    \tau:\LambdaT^r_{n|m}\cong \LambdaT^r_{m|n},
\]
such that
$$\tau(\sum_{\al\in\mbnn,\,\bt\in\mbnm} l_{\al,\bt}
x_1^{\al_1}\dots x_n^{\al_n}y_1^{\bt_1}\dots y_m^{\bt_m})=
\sum_{\al\in\mbnn,\,\bt\in\mbnm} l_{\al,\bt}
y_1^{\al_1}\dots y_n^{\al_n}x_1^{\bt_1}\dots x_m^{\bt_m}.$$
For $k\geq 1$ we have
\[\begin{split}
\tau(\hp_k(x_1,\dots,x_n;y_1,\dots,y_m))&=p_k(y_1,\dots,y_n)+(-1)^{k-1}p_k(x_1,\dots,x_m)\\
&=(-1)^{k-1}\hpk.\end{split}\]
It follows that for $\la\in\msPr$
$$\tau(\hp_\la(x_1,\dots,x_n;y_1,\dots,y_m))=(-1)^{r-\len(\lambda)}\hpla.$$
Since $m>n$, the set  $\{\hp_\la(x_1,\dots,x_n;y_1,\dots,y_m)\mid\la\in\Hnmr\}$
forms a $\mbc$-basis of $\LambdaT^r_{n|m}$ by the already proved case.
Applying $\tau$ shows that the set $ \{\hpla \mid \lambda\in H(n|m,r)\}$
is a basis of $\LambdaT^r_{m|n}$.
\end{proof}
\begin{Rem}\label{Rem}
When \(m>n\), choosing \(H(m|n,r)\) on the power-sum side can be invalid. For example, take \((m,n,r)=(2,1,10)\). The set
\(
H(2|1,10)=\{\lambda\in\msP_{10}\mid \lambda_3\le 1\}
\)
yields linearly dependent elements  \(\mathrm{hp}_\lambda:=\hp_\la(x_1,x_2;y_1)\). The corrected set is
\(
H(1|2,10)=\{\lambda\in\msP_{10}\mid\lambda_2\le 2\}.
\)  Indeed, the following nontrivial linear relation holds among the original family:
\begin{equation}\label{linear comb}
\begin{split}
0 ={}& 25\,\mathrm{hp}_{(7,3)} - 75\,\mathrm{hp}_{(7,2,1)} + 50\,\mathrm{hp}_{(7,1^3)} - 175\,\mathrm{hp}_{(6,3,1)}+ 525\,\mathrm{hp}_{(6,2,1^2)} - 350\,\mathrm{hp}_{(6,1^4)}  \\
&- 21\,\mathrm{hp}_{(5,5)}  + 210\,\mathrm{hp}_{(5,4,1)} + 105\,\mathrm{hp}_{(5,3,1^2)} - 1155\,\mathrm{hp}_{(5,2,1^3)}+ 882\,\mathrm{hp}_{(5,1^5)}  \\
&- 525\,\mathrm{hp}_{(4,4,1^2)} + 1225\,\mathrm{hp}_{(4,3,1^3)}  + 525\,\mathrm{hp}_{(4,2,1^4)} - 910\,\mathrm{hp}_{(4,1^6)}- 1225\,\mathrm{hp}_{(3,3,1^4)}  \\
&+ 1050\,\mathrm{hp}_{(3,2,1^5)}  + 220\,\mathrm{hp}_{(3,1^7)} - 525\,\mathrm{hp}_{(2,2,1^6)}  + 180\,\mathrm{hp}_{(2,1^8)} - 36\,\mathrm{hp}_{(1^{10})}.
\end{split}
\end{equation}
To verify this relation, define  for $r\geq 1$
\[
C_r:=(x_1-\hp_1)^r+(x_2-\hp_1)^r-(-y_1-\hp_1)^r,
\]
where $\hp_0=1$ and  for $k\geq 1$, $\hp_k=\hp_k(x_1,x_2;y_1)=x_1^k+x_2^k+(-1)^{k-1}y_1^k$. Note that $$C_r= \sum_{j=0}^{r}(-1)^j\binom{r}{j}
\hp_1^{\,j}\hp_{r-j}.$$
A direct computation shows that the right-hand side of \eqref{linear comb} equals
\(
25C_3C_7-21C_5^2.
\)
One easily checks that
\[
C_3=3AB(A+B),\qquad
C_5=5AB(A+B)Q,\qquad
C_7=7AB(A+B)Q^2,
\]
where
\(
A=x_1+y_1,\
B=x_2+y_1,\
Q=A^2+AB+B^2.
\)
Consequently,
\[
25C_3C_7
 =525A^2B^2(A+B)^2Q^2=
21C_5^2.
\]
Therefore,
\(
25C_3C_7-21C_5^2=0,
\)
confirming the dependence.
This example clearly demonstrates why the index set must be replaced by
$H(1|2,10)$ when $m>n$.
\end{Rem}

\section{Centers of quantum Schur superalgebras}
This section is devoted to constructing explicit bases for the center of the quantum Schur superalgebra $\bfSr$. Our strategy is to transport known bases of the center of the Hecke algebra $\bfHr$ via the Schur--Weyl duality and the corresponding homomorphism $\eta_r$.
Unlike the classical (non-super) setting, the super case requires a doubly symmetric function ring $\Lamn$ with a parity condition. We shall see that the two classical Hecke center bases, due to Geck and Rouquier and to Jones, give rise to two distinct bases for $\msZ(\bfSr)$, with index sets $\Hmnr$ and $\Hpowmnr$ respectively.

To proceed, we clarify the center in the superalgebra setting.  The following subsection establishes that the ordinary center and the supercenter of $\bfSr$ coincide; hence we may safely use the usual commutator.

\subsection{Evenness of the center}
Before constructing explicit bases, we first clarify the notion of the center in the superalgebra setting. For an associative superalgebra $\mathscr A$ over a field,
 we denote by $\bar a\in\mathbb Z_2$ the parity of a homogeneous element $a\in\mathscr A$. One distinguishes the \emph{supercenter}. For $\bar i\in \mathbb Z_2$, define its $\bar i$-th graded component by
\[
\mathscr Z_{\mathrm{s}}(\mathscr A)_{\bar i}
=
\bigl\{z\in \mathscr A_{\bar i}\mid za=(-1)^{\bar i\bar a}az,\ \text{for all homogeneous }a\in\mathscr A\bigr\}.
\]
Then the supercenter is the graded subspace
\[
\mathscr Z_{\mathrm{s}}(\mathscr A)
=
\mathscr Z_{\mathrm{s}}(\mathscr A)_{\bar 0}
\oplus
\mathscr Z_{\mathrm{s}}(\mathscr A)_{\bar 1}.
\]
This should be distinguished from the ordinary center
\[
\mathscr Z(\mathscr A)=\{z\in\mathscr A\mid za=az,\ \text{for all }a\in\mathscr A\},
\]
which does not require homogeneity and ignores the sign factor. In general, $\mathscr Z(\mathscr A)$ may contain odd elements, in which case the two notions differ. In this subsection we prove that for the quantum Schur superalgebra $\bfSr$, the ordinary center contains no odd elements; consequently $\mathscr Z_{\mathrm{s}}(\bfSr)=\mathscr Z(\bfSr)$. This justifies our use of the usual commutator in the definition of the center throughout the remainder of this paper.

For $1\leq i<m+n$ and $1\leq j\leq m+n$, let
$$\sfe_i=\Zr(\ttE_i),\quad\sff_i=\Zr(\ttF_i),\quad \sfk_j=\Zr(\ttK_j)$$
where $\Zr$ is defined in \eqref{Zr}.
Let $\Sqmnrz$ be the subalgebra of $\Sqmnr$ generated by the elements $\sfk_j^{\pm 1}$ for $1\leq j\leq m+n$.
\begin{Lem}\label{even element}
If $z\in\Sqmnr$ is such that $z=\sfk_hz\sfk_h^{-1}$ for $1\leq h\leq m$, then $z\in\Sqmnr_{\bar 0}$.
\end{Lem}
\begin{proof}
Let $\{x_iz_ky_j\mid i\in I,\,j\in J,\,k\in K\}$ be a $\mbc(\up)$-basis for $\Sqmnr$, where $x_i$
is a monomial in the generators $\sfe_s$, $y_j$ is a monomial in the  generators $\sff_t$ and $z_k\in\Sqmnrz$.
For a monomial $x_i=\sfe_{i_1}\sfe_{i_2}\dots \sfe_{i_{a}}$ define
$$a_{i,p}=\sum_{1\leq s\leq a}\dt_{i_s,p}$$ for $1\leq p< m+n$.
Similarly for $y_j=\sff_{j_1}\sff_{j_2}\dots \sff_{j_{b}}$  define
$$b_{j,p}=\sum_{1\leq u\leq b}\dt_{j_u,p}$$ for $1\leq p< m+n$.
Write $$z=\sum_{i\in I,\,j\in J,\,k\in K}c_{i,j,k}x_iz_ky_j\in \Sqmnr $$ with $c_{i,j,k}\in\mbc(\up)$. Then for $1\leq h\leq m $
we have $$z=\sfk_hz\sfk_h^{-1}=\sum_{i\in I,\,j\in J,\,k\in K}
c_{i,j,k}v_h^{l_{i,j,h}} x_iz_ky_j  $$
where $l_{i,j,1}=\sum_{1\leq p< m+n}(a_{i,p}-b_{j,p}) \dt_{1,p} =a_{i,1}-b_{j,1}$ and
$$l_{i,j,h}=\sum_{1\leq p< m+n}(a_{i,p}-b_{j,p})(\dt_{h,p}-\dt_{h-1,p}) =a_{i,h}-a_{i,h-1}-b_{j,h}+b_{j,h-1}$$
for $1<h\leq m $.
It follows that $c_{i,j,k}=c_{i,j,k}v_h^{l_{i,j,h}}$ for all $i,j,k,h$. Whenever $c_{i,j,k}\not=0$ we must have $l_{i,j,h}=0$ for every $h$. Summing these equalities over
$1\leq h\leq m$ yields
 $$0=\sum_{1\leq h\leq m}l_{i,j,h}=a_{i,m}-b_{j,m}.$$
 Hence, in each non-zero summand  $x_iz_ky_j$, the numbers of $\sfe_m$'s equals the number of $\sff_m$'s. The only odd simple generators of $\bfSr$ are $\sfe_m$ and $\sff_m$; therefore each such monomial is an even element of the superalgebra.
Consequently $z$  itself is an even element.
\end{proof}
The following consequence is immediate from Lemma \ref{even element}.
\begin{Coro}
The ordinary center $\msZ(\Sqmnr)$ of the superalgebra $\Sqmnr$ consists of only even elements. Consequently, we have
\(
\mathscr Z_{\mathrm{s}}(\bfSr) = \mathscr Z(\bfSr).\)
\end{Coro}

\subsection{Geck--Rouquier and Jones bases of $\msZ(\bfHr)$}

For $\la\in\msPr$, let $C_\la$ be the conjugacy
class of $\fS_r$ consisting of permutations of cycle type $\la$, and choose an element $w_\la\in C_\la$ of minimal length. A standard result (see \cite[Th. 5.1]{Ram} and also \cite[Th. 8.2.3]{GP})  states  that for  $w\in\fS_r$ there exists $f_{w,\la}\in\mbz[v,v^{-1}]$, independent of the particular choice of $w_\la$, such that
$$T_w\equiv \sum_{\la\in\msPr}f_{w,\la}T_{w_\la}\mod[\bfHr,\bfHr].$$
Using these coefficients, one defines
\begin{equation}\label{fla}
f_\la^*=\sum_{w\in\fS_r}q^{-\ell(w)}f_{w,\la}T_{w^{-1}}.
\end{equation}
The following fundamental theorem, due to Geck-Rouquier \cite{GR}, provides an explicit basis for the center of the Hecke algebra.

\begin{Thm}\label{GR basis}
The set $\{f_\la^*\mid\la\in\msPr\}$ forms a basis of the center $\msZ(\bfHr)$ of the Hecke algebra $\bfHr$.
\end{Thm}

A second  basis of $\msZ(\bfHr)$ is due to Jones. Let $\la$ be a composition of $r$ with length $l$. The corresponding parabolic subalgebra is $\bsH_\la=\bsH_{\la_1}\bsH_{\la_2}\dots\bsH_{\la_l}$
where each factor $\bsH_{\la_i}$ is isomorphic to
$\bsH_v(\fS_{\la_i})$.
Let $\fS_\la$ be the standard Young subgroup  attached to this composition, and define the set of distinguished right coset representatives
$$\msD_\la=\{d\mid d\in\fSr,\ell(wd)=\ell(w)+\ell(d)\text{ for
$w\in\fS_\la$}\}.$$
For any element $h\in\bsH_\la$ set
$$\msN_\la(h)=\sum_{w\in\msD_\la}q^{-\ell(w)}T_whT_{w^{-1}}$$
(see \cite{Jones}).
Now take $\la$ to be a partition of $r$ with $\ell(\la)=l$. Define
  $F_\la=f_{\la_1}^*f_{\la_2}^*\dots f_{\la_l}^*$,
where each $f_{\la_i}^*$ is the Geck-Rouquier central element in the factor $\bsH_{\la_i}$ corresponding to the one-part partition $(\la_i)$.
The subsequent theorem, proved by Jones  \cite{Jones} (cf. \cite{WW}),
provides the second basis.

\begin{Thm}
The set $\{\msN_\la(F_\la)\mid\la\in\msPr\}$ constitutes a basis of the center $\msZ(\bfHr)$ of the Hecke algebra $\bfHr$.
\end{Thm}

These two bases will serve as the starting point for our construction.   The next step is to connect these bases with symmetric function theory, which will enable us to transfer them to the quantum Schur superalgebra via the duality.

\subsection{An isomorphism from $\msZ(\bfHr)$ to symmetric functions}
To transfer the known central bases of the Hecke algebra $\bfHr$ to the quantum Schur superalgebra, we first need a bridge between the center of the Hecke algebra $\bfHr$ and the ring of symmetric functions.
Let $\Lav=\La\ot_\mbc\mbc(v)$ and $\Lamnv=\tiLamn\ot_\mbc\mbc(v)$.
For $r \in \mbn$, we denote by
$$\Lar_v=\Lar \otimes_\mbc \mbc(v),\quad\Lamnrv=\Lamnr\ot_\mbc\mbc(v)$$
the respective homogeneous degree-$r$ components.
By base change, the ring homomorphism $\Phimn$ defined in \eqref{Phimn} extends to a $\mathbb C(v)$-linear homomorphism from $\Lav$ to $\Lamnv$;
we shall denote this extension again  by $\Phimn$.
For $\lambda \in \Par$, define:
$$
\widetilde{s}_\lambda(x)=\frac{1}{\kappa_\lambda(q)} \tts_\lambda(x) \in \Lar_v,\quad
\widetilde{p}_\lambda(x)=(q-1)^{\ell(\lambda)} \prod_{1 \leqslant i \leqslant \ell(\lambda)} \frac{1}{q^{\lambda_i}-1} p_\lambda(x) \in \Lar_v
$$
Here $\kappa_\lambda(q) \in \mathbb{Z}[q,q^{-1}]$ is the Schur element for $\bfHr$ (see \cite[Th. 7.2.1]{GP}), and $p_\la(x)$ is defined in \eqref{pla}.

For $\lambda \in \Par$, we write
$$
m_\lambda(x)=\sum_{\mu \in \Par} L_{\lambda, \mu} p_\mu(x)
$$
where $L_{\lambda, \mu} \in \mbc$. For $\lambda \in \Par$ let
\begin{equation}\label{tilde m_lambda(x)}
\widetilde{m}_\lambda(x)=\sum_{\mu \in \Par} L_{\lambda, \mu} \frac{(q-1)^{\ell(\lambda)}}{\prod_{1 \leqslant s \leqslant \ell(\mu)}\left(q^{\mu_s}-1\right)} p_\mu(x) \in \Lar_v.
\end{equation}

For $\la\in\msPr$ let
$\epvla$ be the central primitive idempotent in $\bfHr$   acting on the Specht module $S_\up^\mu$ as  $\delta_{\lambda, \mu}$ for $\mu \in \Par$. By \cite[Lem. 2.2]{WW} we have
\begin{equation}\label{vep la}
\epvla=\frac{1}{\kappa_\la(q)}\sum_{\mu\in\msPr} \chi^\la_v(T_{w_\mu})f_\mu^*,
\end{equation}
where $\chi^\la_v$ is the irreducible character of $\bfHr$ such that
the specialization of $\chi^\la_v$ at $v=1$ coincides with the character $\chi^\la$
of the Specht module of $\frak{S}_r$ associated to $\la$.

Wan and Wang \cite[3.7]{WW} established a linear isomorphism
\begin{equation}\label{def psir}
\begin{split}
\psi_r:& \ZHr \longrightarrow \Lar_v\\
&\qquad\qquad  f_\lambda^* \mapsto \widetilde{m}_\lambda(x)\ \text{for $\lambda \in \Par$. }
\end{split}
\end{equation}
Under this map, the central primitive idempotents and the Jones basis elements have explicit images, as summarized in the next theorem.

\begin{Thm}\cite{WW}\label{psi_r}
For $\lambda \in \Par$, we have $\psi_r\left(\epvla\right)=\widetilde{s}_\lambda(x)$ and $\psi_r\left(\mathscr{N}_\lambda\left(F_\lambda\right)\right)=\widetilde{p}_\lambda(x)$.
\end{Thm}
This theorem is the key that allows us to translate linear relations among Hecke center elements into relations among symmetric functions, and hence to transfer the bases to the quantum Schur superalgebra.

\subsection{Transfer of the Jones basis}
We now apply the isomorphism $\psi_r$ to obtain the first central basis of  namely the image of the Jones basis under $\eta_r$.

\begin{Thm}\label{Jones basis}
The set $$\{\eta_r(\mathscr{N}_\lambda(F_\lambda)) \mid \lambda \in \Hpowmnr\}$$ forms a basis of the center $\ZSq$ of the quantum Schur superalgebra $\Sqmnr$.
\end{Thm}
\begin{proof}
 For $\lambda \in \Hmnr$ let $\tepvla$
denote the central primitive idempotent in $\Sqmnr$ whose action on $L_\up(\mu)$  is scalar multiplication by $\delta_{\lambda, \mu}$ for $\mu \in \Hmnr$. By \eqref{Qdecom} and \cite[(3.9.3)]{Fu21}, for any $\lambda \in \Par$,
\begin{equation}\label{eta_r}
  \eta_r(\epvla)= \begin{cases}\tepvla & \text { if } \lambda \in \Hmnr \\ 0 & \text { otherwise }.\end{cases}
\end{equation}

  For $\lambda \in \Par$ write
\begin{equation}\label{N_la}
  \mathscr{N}_\lambda\left(F_\lambda\right)=\sum_{\mu \in \Par} b_{\lambda, \mu} \epvmu
\end{equation}
where $b_{\lambda, \mu} \in \mathbb{C}(\up)$. Applying \eqref{eta_r} gives, for $\lambda \in \Hpowmnr$,
\begin{equation}\label{eta_rN_la}
  \eta_r\left(\mathscr{N}_\lambda\left(F_\lambda\right)\right)=\sum_{\mu \in \Hmnr} b_{\lambda, \mu} \tepvmu .
\end{equation}
On the other hand, by Theorem \ref{psi_r} and \eqref{N_la}, we have, for $\lambda \in \Par$,
$$
\widetilde{p}_\lambda(x)=\sum_{\mu \in \Par} b_{\lambda, \mu} \widetilde{s}_\mu(x).
$$
Applying $\Phimn$ and using Theorem \ref{super Schur function}, we obtain, for $\lambda \in \Hpowmnr$,
$$
\widetilde{\hp}_\lambda (x_1, \dots, x_m;y_1, \dots, y_n )=\sum_{\mu \in \Hmnr} b_{\lambda, \mu} \widetilde{\hs}_\mu (x_1, \dots, x_m; y_1,\dots,y_n ),
$$
where $\widetilde{\hp}_\lambda (x_1, \dots, x_m;y_1, \dots, y_n )=\Phimnv (\widetilde{p}_\lambda(x))$ and $\widetilde{\hs}_\mu (x_1, \dots, x_m; y_1,\dots,y_n)=\Phimnv(\widetilde{\hs}_\mu(x))$.
By Theorems \ref{super Schur function} and \ref{thm:corrected-power-product} that the matrix $$B= \left(b_{\lambda, \mu}\right)_{\lambda\in\Hpowmnr, \mu \in \Hmnr}$$ is invertible. Hence by \eqref{eta_rN_la} the set $\left\{\eta_r\left(\mathscr{N}_\lambda\left(F_\lambda\right)\right) \mid \lambda \in \Hpowmnr\right\}$ forms a basis of $\ZSq$, because $\left\{\tepvla \mid \lambda \in \Hmnr\right\}$ is already a basis of $\ZSq$.
\end{proof}

\subsection{Transfer of the Geck--Rouquier basis via integral forms}
In this subsection we prove that the image of the Geck--Rouquier basis $\{f_\lambda^*\}$ under the duality map $\eta_r$ gives a basis of the center of the quantum Schur superalgebra, indexed by the set $\Hpow(m|n,r)$. The argument uses an integral form of the Hecke algebra and specialization at $v=1$ to reduce the invertibility of the transition matrix to the classical symmetric group case.

Let \(
\R=\mathbb{C}[v]_{(v-1)}\) and define
\[
{\mathcal H}_\R(\frak{S}_r)=\bigoplus_{w\in\frak{S}_r} \R\,T_w \subseteq {\mathcal H}_v(\frak{S}_r).
\]
Denote by \(\mathscr Z(\mathcal H_\R(\frak S_r))\) the center of \({\mathcal H}_\R(\frak{S}_r)\). We first record a standard flatness fact.

\begin{Lem}\label{flat}
For any commutative ring \(A\) and any multiplicative subset \(S\), the localization \(S^{-1}A\) is a flat \(A\)-module.
\end{Lem}
\begin{proof}
By \cite[{Prop 3.3 \& 3.5}]{AM}, the localization functor \(S^{-1}\) is exact, and for every \(A\)-module \(M\) there is a natural isomorphism
\(
S^{-1}M \cong M \otimes_A S^{-1}A.
\)
Hence the tensor product functor $-\otimes_A S^{-1}A$ is exact, which by definition means that $S^{-1}A$ is flat over $A$.
\end{proof}
Using flatness, we can compare the $\R$-center with its base changes.
\begin{Prop}\label{rank}
Let \(p=|\mathscr P_r|\). Then:
\begin{enumerate}
\item \(\operatorname{rank}_\R \mathscr Z(\mathcal H_\R(\frak S_r)) = \dim_{\mathbb{C}(v)} \mathscr Z({\mathcal H}_v(\frak{S}_r))=p.\)
\item There is a canonical isomorphism \(\mathscr Z(\mathcal H_\R(\frak S_r))\otimes_\R \mathbb C \cong \mathscr Z(\mathbb C\frak{S}_r)\), where \(\mathbb C\) is viewed as an \(\R\)-module via \(v\mapsto 1\).
\end{enumerate}
\end{Prop}

\begin{proof}
By \cite[Chapter III Exercise 6]{Hung}  \(\R\) is  a principal ideal domain.
Hence, as $\mathcal{H}_{\R}(\mathfrak{S}_{r})$ is free over $\R$, its center is free of finite rank.
 Consider the \(\mathbb{C}(v)\)-linear map
\[
f:\mathscr Z(\mathcal H_\R(\frak S_r))\otimes_\R \mathbb{C}(v)\longrightarrow \mathscr Z({\mathcal H}_v(\frak{S}_r)),\qquad z\otimes a \longmapsto a z.
\]
By Theorem \ref{GR basis}  \(f\) is surjective.
By Lemma \ref{flat}, \(\mathbb{C}(v)\) is a flat \(\R\)-module.
Hence the inclusion \(\mathscr Z(\mathcal H_\R(\frak S_r))\hookrightarrow {\mathcal H}_\R(\frak{S}_r)\) induces an injective map
\[\mathscr Z(\mathcal H_\R(\frak S_r))\ot_\R\mbc(v)\hookrightarrow {\mathcal H}_\R(\frak{S}_r)\ot_\R\mbc(v)\cong\bfHr.\]
Hence \(f\)  is injective as well, and therefore an isomorphism.
Consequently,
\[
\operatorname{rank}_\R \mathscr Z(\mathcal H_\R(\frak S_r))=\dim_{\mathbb{C}(v)}(\mathscr Z(\mathcal H_\R(\frak S_r))\otimes_\R \mathbb{C}(v))=\dim_{\mathbb{C}(v)}\mathscr Z({\mathcal H}_v(\frak{S}_r))=p.
\]
It follows that
\(
\dim_{\mathbb C}(\mathscr Z(\mathcal H_\R(\frak S_r))\otimes_\R \mathbb C)=p=\dim_{\mathbb C}\mathscr Z(\mathbb C\frak{S}_r).
\) Now define a \(\mathbb C\)-linear map
\[
g:\mathscr Z(\mathcal H_\R(\frak S_r))\otimes_\R \mathbb C \longrightarrow \mathbb C\frak{S}_r,\qquad z\otimes 1 \longmapsto z|_{v=1},
\]
induced by the composite \(\mathscr Z(\mathcal H_\R(\frak S_r))\otimes_\R \mathbb C \rightarrow {\mathcal H}_\R(\frak{S}_r)\otimes_\R \mathbb C \cong \mathbb C\frak{S}_r\), where \(T_w\otimes 1\mapsto w\). Its image lies in the center $\mathscr Z(\mathbb C\frak{S}_r)$. Moreover, the Geck--Rouquier basis \( f_\lambda^* \) specialises at \(v=1\) to the basis of conjugacy class sums, which spans \(\mathscr Z(\mathbb C\frak{S}_r)\); hence the image of \(g\) is $\mathscr Z(\mathbb C\frak{S}_r)$. By the dimension equality, we obtain the isomorphism
\(\mathscr Z(\mathcal H_\R(\frak S_r))\otimes_\R \mathbb C \cong \mathscr Z(\mathbb C\frak{S}_r)\).
\end{proof}
We now show that the primitive central idempotents of the Hecke algebra lie in the integral form and form an $\R$-basis of its center.
\begin{Prop}\label{basis of C(H_R)}
The set \(\{\epvla\mid\la\in\mathscr P_r\}\) is an \(\R\)-basis of \(\mathscr Z(\mathcal H_\R(\frak S_r))\).
\end{Prop}
\begin{proof}
By \cite[(2.4)]{WW} we have
$$\kappa_\la(1)=\frac{r!}{d_\la}\not=0,$$
where $d_\la$ is the degree of the irreducible character $\chi^\la$.
Hence ${1}/{\kappa_\la(q)}\in \R$.
Furthermore, by \cite[Th. 4.14]{Ram}, $\chi^\la_v(T_{w_\mu})\in\mbz[v,v^{-1}]\han \R$.
Thus  each \(\epvla \) lies in \(\mathscr Z(\mathcal H_\R(\frak S_r))\).
Set \(\mathfrak m=(v-1)\). Then
\[
\mathscr Z(\mathcal H_\R(\frak S_r))/\mathfrak m\mathscr Z(\mathcal H_\R(\frak S_r)) \cong
\mathscr Z(\mathcal H_\R(\frak S_r))\otimes_\R \R/{\frak m}
\cong \mathscr Z(\mathcal H_\R(\frak S_r))\otimes_\R \mathbb C.
\]
It follows from Proposition~\ref{rank}(2) that
\(
\mathscr Z(\mathcal H_\R(\frak S_r))/\mathfrak m\mathscr Z(\mathcal H_\R(\frak S_r)) \cong \mathscr Z(\mathbb C\frak{S}_r).
\)
Under this isomorphism,  \(\epvla+\mathfrak m\mathscr Z(\mathcal H_\R(\frak S_r))\) corresponds to \(\epvla|_{v=1}\). From \cite[(2.5)]{WW}
\[
\epvla|_{v=1}
=\frac{d_\la}{r!}\sum_{\mu\in\msPr}\chi^\la(w_\mu)c_\mu
=\epla ,
\]
where $c_\mu$ is the class sum and \(\epla\) is the primitive central idempotent of \(\mathbb C\frak{S}_r\) for \(\la\). The set \(\{\epla\mid \la\in\mathscr P_r\}\) form a basis of \(\mathscr Z(\mathbb C\frak{S}_r)\); therefore the elements \(\epvla +\mathfrak m\mathscr Z(\mathcal H_\R(\frak S_r))\) ($\la\in\msPr$)  generate the quotient
\(
\mathscr Z(\mathcal H_\R(\frak S_r))/\mathfrak m\mathscr Z(\mathcal H_\R(\frak S_r)).
\)
By Proposition \ref{rank},
\(\mathscr Z(\mathcal H_\R(\frak S_r))\) is finitely generated over $\R$.
By Nakayama's lemma, the set \(\{\epvla\mid\la\in\mathscr P_r\}\) generates \(\mathscr Z(\mathcal H_\R(\frak S_r))\) as an \(\R\)-module.
Finally, these generators are linearly independent over \(\R\), so they form an \(\R\)-basis.
\end{proof}
We now transfer the Geck--Rouquier basis to the center of the quantum Schur superalgebra. The following theorem shows that the images of these elements form a basis indexed by the set \(\Hpow(m|n,r)\).
\begin{Thm}\label{basis of Sqmnr}
The set $\left\{\eta_r\left(f_\lambda^*\right) \mid \lambda \in \Hpowmnr\right\}$ forms a basis of the center $\ZSq$ of the quantum Schur superalgebra $\Sqmnr$.
\end{Thm}
\begin{proof}
With the \(\R\)-basis of   Proposition \ref{basis of C(H_R)} at hand,
write each Geck--Rouquier element as
\begin{equation}\label{f_la}
  f_\lambda^*=\sum_{\mu \in \Par} a_{\lambda, \mu}\epvmu ,\qquad a_{\lambda,\mu}\in \R.
\end{equation}
Applying \eqref{eta_r}, we obtain for $\lambda \in \msPr$
\begin{equation}\label{eta_rf_la}
  \eta_r\left(f_\lambda^*\right)=\sum_{\mu \in \Hmnr} a_{\lambda, \mu} \tepvmu.
\end{equation}

It remains to show that the matrix \(A=(a_{\lambda,\mu})_{\lambda\in \Hpow,\ \mu\in H}\) (where \(H=H(m|n,r),\,\Hpow=\Hpowmnr\)) is invertible over \(\mbc(v)\). To see this, we specialize the identity \(f_\lambda^*=\sum_\mu a_{\lambda,\mu}\epvmu \) at \(v=1\). The Geck--Rouquier element \(f_\lambda^*\) specializes to the conjugacy class sum \(c_\lambda\), and each primitive idempotent \(\epvmu\) specializes to the primitive central idempotent \(\epmu\) of \(\mathbb{C}\mathfrak{S}_r\). Hence
\begin{equation}\label{clambda express}
c_\lambda=\sum_{\mu\in\msPr} a_{\lambda,\mu}(1)\,\epmu ,\qquad \lambda\in\msPr.
\end{equation}
By \cite[(4.7) \& (5.1)]{WW} we have
\[\frac 1{z_\la} p_\la(x)
=\sum_{\mu \in \Par} a_{\lambda, \mu}(1)\frac{d_\mu}{r!} s_\mu(x)\]
where $z_\la=\prod_{1\leq i\leq k}i^{m_i}m_i!$ for $\la=(k^{m_k},\cdots,2^{m_2},1^{m_1})\in\msPr$ and
$d_\la$ is the degree of the irreducible character $\chi^\la$.
Applying $\Phimn$
  to both sides  and using Theorem \ref{super Schur function}  (which implies that $\Phimn(s_\mu(x))=0$ for $\mu\not\in H$), we obtain
\[\frac 1{z_\la} \hp_\la(x_1,\dots,x_m;y_1,\dots,y_n)=
\sum_{\mu \in \Hmnr} a_{\lambda, \mu}(1)\frac{d_\mu}{r!} \hs_\mu(x_1,\dots,x_m;y_1,\dots,y_n).\]
By  Theorems  \ref{super Schur function} and \ref{thm:corrected-power-product},
 the determinant
\begin{equation}\label{det}
\det\left(\frac{d_\mu}{r!}a_{\lambda,\mu}(1)\right)_{\lambda\in \Hpow,\ \mu\in H}\neq 0.
\end{equation}
Since \(d_\mu/r!\neq 0\), we have \(\det A(1)\neq 0\), where \(A(1)=(a_{\lambda,\mu}(1))_{\lambda\in \Hpow,\ \mu\in H}\). Hence \(\det A\neq 0\) in \(\R\).
Therefore  $A$ is invertible over $\mbc(v)$, and by  \eqref{eta_rf_la},  the elements $\eta_r(f_\lambda^*)$ for $\lambda \in \Hpowmnr$ form a basis of $\ZSq$.
\end{proof}

\subsection{Transfer of the Geck--Rouquier basis via power-series specialization}
We now give an alternative transfer of the Geck--Rouquier basis, which yields a basis indexed by the full set of $(m|n)$-hook partitions. This method uses an embedding $\mathbb{C}(q)\hookrightarrow \mathbb{C}((t))$, $q\mapsto t^{-1}$, and a general lemma about bases of power series.

\begin{Lem} \label{lem:t-adic-basis}
Let $M$ be an $N$-dimensional vector space over $\C$.
\begin{enumerate}
\item[(1)] If
$
u_1(t),\ldots,u_N(t)\in M\otimes_{\mathbb C}\mathbb C[[t]]$
and their specializations \(u_1(0),\ldots,u_N(0)\) form a basis of \(M\), then the \(u_i(t)\) form a basis of \(M\otimes_{\mathbb C}\mathbb C((t))\) over \(\mathbb C((t))\).

\item[(2)] Assume further  that the \(u_i(t)\) are obtained from elements
\[
\widetilde u_1,\ldots,\widetilde u_N\in M\otimes_{\mathbb C}\mathbb C(q)
\]
via the embedding \(\mathbb C(q)\hookrightarrow \mathbb C((t))\), \(q\mapsto t^{-1}\). Then the \(\widetilde u_i\) form a basis of \(M\otimes_{\mathbb C}\mathbb C(q)\).
\end{enumerate}
\end{Lem}

\begin{proof}
Choose a basis of \(M\), and let \(U(t)\) be the \(N\times N\) matrix whose columns are the coordinate vectors of the \(u_i(t)\). Since \(u_1(0),\ldots,u_N(0)\) form a basis of \(M\), we have \(\det U(0)\neq 0\). Hence \(\det U(t)\) is a nonzero element of \(\mathbb C[[t]]\), so it is nonzero in \(\mathbb C((t))\). Therefore $u_1(t),\ldots,u_N(t)$ are linearly independent over \(\mathbb C((t))\), and since there are exactly \(N\) of them, they form a basis of \(M\otimes_{\mathbb C}\mathbb C((t))\). This proves (1).

For (2), let \(D(q)\in \mathbb C(q)\) be the determinant of the matrix formed by the coordinate vectors of the \(\widetilde u_i\) with respect to the same basis of \(M\). Under the embedding \(q\mapsto t^{-1}\), the image of \(D(q)\) is exactly \(\det U(t)\), which is nonzero in \(\mathbb C((t))\). Since the embedding is injective, \(D(q)\neq 0\) in \(\mathbb C(q)\). Thus the \(\widetilde u_i\) are linearly independent over \(\mathbb C(q)\) and hence form a basis.
\end{proof}

Recall the elements
$\mathrm{hm}_\lambda(x_1,\dots,x_m;y_1,\dots,y_n)$ defined in \eqref{hmla}.
For $\lambda\in\msPr $ let $$\ti{\hm}_\la(x_1,\dots,x_m;y_1,\dots,y_n)=\Phi_{m|n}(\widetilde m_\lambda(x)),$$
where $\widetilde m_\lambda(x)$ is defined in \eqref{tilde m_lambda(x)}.
\begin{Lem} \label{lem:modified-specialization}
For $\lambda\in\msPr $, set $t=q^{-1}$ and
\[
        G_\lambda(t):=q^r(q-1)^{-\ell(\lambda)}\ti{\hm}_\la(x_1,\dots,x_m;y_1,\dots,y_n).
\]
Then $G_\lambda(t)\in \Lambda^r_{m|n}\otimes_\C\C[[t]]$ and
$G_\lambda(0)=\mathrm{hm}_\lambda(x_1,\dots,x_m;y_1,\dots,y_n)$.
\end{Lem}

\begin{proof}
Using the definition of $\widetilde m_\lambda(x)$, we get
\[
\begin{aligned}
        G_\lambda(t)
        &=\sum_{\mu\in\msPr}L_{\lambda\mu}
          \frac{q^r}{\prod_{1\leq i\leq\ell(\mu)}(q^{\mu_i}-1)}\Phi_{m|n}(p_\mu(x))  \\
        &=\sum_{\mu\in\msPr}L_{\lambda\mu}
          \prod_{1\leq i\leq\ell(\mu)}(1-t^{\mu_i})^{-1}\Phi_{m|n}(p_\mu(x)).
\end{aligned}
\]
Each factor $(1-t^{\mu_i})^{-1}$ lies in $\C[[t]]$ and has constant term $1$.  Therefore $G_\lambda(t)$ lies in $\Lambda^r_{m|n}\otimes_\C\C[[t]]$ and
\[
        G_\lambda(0)=\sum_{\mu\in\msPr}L_{\lambda\mu}\Phi_{m|n}(p_\mu(x))
        =\Phi_{m|n}(m_\lambda(x))=\mathrm{hm}_\lambda(x_1,\dots,x_m;y_1,\dots,y_n) .
\]
This completes the proof.
\end{proof}

The preceding lemma provides the necessary specialization information. Applying Lemma \ref{lem:t-adic-basis}(2) to the family $G_\la(t)$ immediately gives the following basis theorem.

\begin{Thm} \label{thm:modified-monomial-basis}
For every $r\ge 0$,
\[
        \{\ti{\hm}_\la(x_1,\dots,x_m;y_1,\dots,y_n)\mid \lambda\in H(m|n,r)\}
\]
forms a basis of $\Lambda^r_{m|n}\otimes_\C\mbc(v)$.
\end{Thm}

\begin{proof}
Apply Lemma \ref{lem:t-adic-basis}(2) with \(M=\Lamnr\) and with the elements
\[
\widetilde u_\lambda:=G_\la(t)\in M\otimes_{\mathbb C}\mathbb C(q).
\]
By Lemmas \ref{hm_lambda} and \ref{lem:modified-specialization}, the set \(\{G_\la(0)\mid \lambda\in H(m|n,r)\}\) is a \(\mathbb C\)-basis of \(M\).
It follows from Lemma \ref{lem:t-adic-basis}(2) that the set \(\{G_\lambda(t)\mid\la\in\Hmnr\}\) form a basis of \(M\otimes_{\mathbb C}\mathbb C(q)\). Therefore the set
$\{\ti{\hm}_\la(x_1,\dots,x_m;y_1,\dots,y_n)\mid \lambda\in H(m|n,r)\}$
forms a basis of $\Lambda^r_{m|n}\otimes_\C\mbc(v)$.
\end{proof}
Finally, we are in a position to transfer the Geck--Rouquier basis.
\begin{Thm}\label{GR basis of quantum Schur superalgebra}
The set $\left\{\eta_r\left(f_\lambda^*\right) \mid \lambda \in \Hmnr\right\}$ forms a basis of the center $\ZSq$ of the quantum Schur superalgebra $\Sqmnr$.
\end{Thm}

\begin{proof}
By \eqref{def psir}, Theorem \ref{psi_r}, and \eqref{f_la} we have
$$
\widetilde{m}_\lambda(x)=\sum_{\mu \in \Par} a_{\lambda, \mu} \widetilde{s}_\mu(x)
$$
for $\lambda \in \Par$.
Applying $\Phimn$ and using Theorem \ref{super Schur function} gives, for $\lambda \in \Hmnr$,
$$
\widetilde{\hm}_\lambda\left(x_1,\dots, x_m; y_1, \dots, y_n\right)=\sum_{\mu \in \Hmnr} a_{\lambda, \mu} \widetilde{\hs}_\mu\left(x_1,\dots, x_m; y_1, \dots, y_n\right).
$$
By  Theorems  \ref{super Schur function} and \ref{thm:modified-monomial-basis}, the matrix $A=\left(a_{\lambda, \mu}\right)_{\lambda, \mu \in \Hmnr}$ is invertible. Consequently, from \eqref{eta_rf_la}, the elements $\eta_r(f_\lambda^*)$ for $\lambda \in \Hmnr$ form a basis of $\ZSq$.
\end{proof}

\subsection{The center of the classical Schur superalgebra}

In this final subsection of Section 4, we specialise the quantum parameter \(v\) to \(1\) and describe bases for the center of the classical Schur superalgebra \(\mathcal{S}(m|n,r)\).

Let \(\mathfrak{gl}_{m|n}=\mathfrak{gl}_{m|n}(\mathbb{C})\) be the general linear Lie superalgebra over \(\mathbb{C}\), with standard basis \(E_{ij}\) (\(1\le i,j\le m+n\)) and \(\mathbb{Z}_2\)-grading \(\bar E_{ij}=\bar i+\bar j\). Let \(\mathcal{U}(\mathfrak{gl}_{m|n})\) be its universal enveloping algebra, and let \(\zrcl:\mathcal{U}(\mathfrak{gl}_{m|n})\to \operatorname{End}_{\mathbb{C}}(V^{\otimes r})\) be the natural representation on the \(r\)-fold tensor power of the natural module \(V=\mathbb{C}^{m|n}\). The classical Schur superalgebra is defined as the image
\[
\mathcal{S}(m|n,r):=\zrcl(\mathcal{U}(\mathfrak{gl}_{m|n}))\subseteq \operatorname{End}_{\mathbb{C}}(V^{\otimes r}).
\]
 Let \(\etarcl:\mathbb{C}\mathfrak{S}_r^{\mathrm {op}} \to \operatorname{End}_{\mathbb{C}}(V^{\otimes r})\) be the signed permutation action of the symmetric group.

The conjugacy class sums \(\{c_\lambda\mid \lambda\in\msPr\}\) form the standard basis of the center of  \(\mathbb{C}\mathfrak{S}_r\). Applying the  map \(\etarcl\) to these elements, gives the following classical analogue of Theorem \ref{basis of Sqmnr}.

\begin{Thm}\label{classical case}
The set
\[
\{\etarcl(c_\lambda) \mid \lambda \in \Hpow(m|n,r)\}
\]
forms a \(\mathbb{C}\)-basis of the center \(\mathscr{Z}(\mathcal{S}(m|n,r))\) of the classical Schur superalgebra \(\mathcal{S}(m|n,r)\).
\end{Thm}

\begin{proof}
Applying the homomorphism \(\etarcl\) to the identity \eqref{clambda express}  yields
\[
\etarcl(c_\lambda)=\sum_{\mu\in\Hmnr} a_{\lambda,\mu}(1) \tepmu,
\]
where $\tepmu$
denote the central primitive idempotent in $\Smnr$.
 By \eqref{det}, the matrix
\(
\bigl(a_{\lambda,\mu}(1)\bigr)_{\lambda\in \Hpowmnr,\ \mu\in\Hmnr}
\)
is invertible. Therefore the elements \(\etarcl(c_\lambda)\) with \(\lambda\in \Hpow(m|n,r)\) form a basis of \(\mathscr{Z}(\mathcal{S}(m|n,r))\).
\end{proof}


\begin{thebibliography}{99}\frenchspacing
\bibitem{AM}
M. F.  Atiyah and I. G.  Macdonald, \textit{Introduction to commutative algebra}, Addison-Wesley Publishing Co., Reading, Mass.-London-Don Mills, Ont., 1969.

\bibitem{BY}
P. Batra and H. Yamane, \textit{Centers of generalized quantum groups}, J. Pure Appl. Algebra {\bf 222} (2018), 1203--1241.


\bibitem{BR}
A. Berele and A. Regev,\textit{Hook Young diagrams with applications to combinatorics and
              to representations of Lie superalgebras}, Adv. in Math. {\bf 64}, (1987), 118--175.
\bibitem{CW}
S. Cheng and W. Wang, \textit{ Dualities and representations of Lie superalgebras}, Graduate Studies in Mathematics, {\bf 144}, American Mathematical Society, Providence, RI, 2012.
\bibitem{EK}
H. El Turkey and J. Kujawa, \textit{Presenting Schur superalgebras}, Pacific J. Math. {\bf 262} (2013), no. 2, 285--316.
\bibitem{Fu21}
Q. Fu, \textit{ Schur-Weyl duality and centers of quantum Schur algebras}, J. Pure Appl. Algebra, {\bf 225}(2021), no. 5, 106594.





\bibitem{GP}
M. Geck and G. Pfeiffer, \textit{Characters of finite Coxeter groups and Iwahori-Hecke algebras}, London Mathematical Society Monographs. New Series, {\bf 21}. The Clarendon Press, Oxford University Press, New York, 2000.


\bibitem{GR}
M. Geck and R. Rouquier,
{\em Centers and simple modules for Iwahori-Hecke algebras}, Finite reductive groups (Luminy, 1994), 251--272, Progr. Math., 141, Birkh$\mathrm{\ddot{a}}$user Boston, Boston, MA, 1997.

\bibitem{GePr}
T. Geetha and A. Prasad,
{\em Center of the Schur algebra}, Asian-Eur. J. Math. {\bf 9} (2016), 1650006, 11 pp.


\bibitem{Hung}
T. W. Hungerford, \textit{Algebra}, Reprint of the 1974 original. Graduate Texts in Mathematics, {\bf 73}. Springer-Verlag, New York-Berlin, 1980.

\bibitem{Jones}
L. Jones, \textit{Centers of generic Hecke algebras}, Trans. Amer. Math. Soc. {\bf 317} (1990), 361--392.

\bibitem{KempRamgoolam2020}
G. Kemp and S. Ramgoolam, \textit{
BPS states, conserved charges and centres of symmetric group algebras}
J. High Energy Phys. 2020, {\bf 146}, 42 pp.

\bibitem{LWY}
Y. Luo, Y.J. Wang, Y. Ye, \textit{On the Harish-Chandra homomorphism for quantum superalgebras}, Comm. Math. Phys. {\bf 393} (2022),  1483--1527.

\bibitem{Macd}
I. G. Macdonald, {\em Symmetric functions and Hall polynomials}, Second edition. With contributions by A. Zelevinsky.  Oxford Mathematical Monographs, The Clarendon Press, Oxford University Press, New York, 1995.

\bibitem{Mi}
H. Mitsuhashi, \textit{Schur-Weyl reciprocity between the quantum superalgebra and
              the {I}wahori-{H}ecke algebra}, Algebr. Represent. Theory {\bf 9} (2006), 309--322.
\bibitem{Mi1}
H. Mitsuhashi, \textit{A super Frobenius formula for the characters of Iwahori-Hecke algebras}, Linear Multilinear Algebra {\bf 58} (2010),  941--955.

\bibitem{Ram}
A. Ram, \textit{A Frobenius formula for the characters of the Hecke algebras}, Invent. Math. {\bf 106} (1991), 461--488.
\bibitem{Se}
A. Sergeev, \textit{The tensor algebra of the identity representation as a module over the Lie superalgebras {${\frak gl}(n,m)$}\ and {$Q(n)$}}, Math. USSR Sbornik {\bf 51} (1985), 419--427.
\bibitem{WW}
J. Wan, W. Wang, \textit{Frobenius map for the centers of Hecke algebras}, Trans. Am. Math. Soc. {\bf 367} (2015), 5507--5520.
\bibitem{ZR}
R. Zhang,
{\em Finite-dimensional irreducible representations of the quantum supergroup $U_q(\frak{gl}(m|n))$},
J. Math. Phys. {\bf 34} (1993), 1236--1254.
\end{thebibliography}
\end{document}